\documentclass[12pt,leqno,amscd, amsfonts, amssymb,pstricks,verbatim]{amsart}
\usepackage{mathrsfs}
\usepackage{amsmath}
\usepackage{hyperref}
\usepackage{bookmark}

\def\a{\alpha}

\def\b{\beta}
\def\d{\Delta}

\def\Z{\mathbb{Z}}

\def\C{\mathbb{C}}

\numberwithin{equation}{section}
\newtheorem{theo}{Theorem}[section]

\newtheorem{lemm}[theo]{Lemma}

\newtheorem{case}{Case}

\newtheorem{subcase}{Subcase}

\allowdisplaybreaks

\begin{document}

\title[2-local derivations]{2-local derivations on the planar Galilean conformal algebra}

\author{Qiu-Fan Chen, Yan He}

\address{Department of Mathematics, Shanghai Maritime University,
 Shanghai, 201306, China.}\email{chenqf@shmtu.edu.cn}
\address{School of Mathematics and Statistics, Changshu Institute of Technology,
Changshu, Suzhou, 215500, China.}\email{heyan913533012@163.com}



\thanks{This work is supported by National Natural Science Foundation of China (Grant Nos. 12271345 and  12101082), Natural Science Foundation for Youths of Jiangsu Province (Grant  Nos. BK20201051) and Jiangsu Provincial Double-Innovation Doctor Program (Grant  Nos. JSSCBS20210742). }

\begin{abstract}
The present paper is devoted to studying 2-local derivations on the planar Galilean conformal algebra.  We prove that
every 2-local derivation on the planar Galilean conformal algebra is a derivation.
\end{abstract}

\maketitle
\setcounter{tocdepth}{1}\tableofcontents
\begin{center}
\end{center}

\section{Introduction}

Throughout the paper, we denote by $\C ,\,\Z,\,\Z^*$ the sets of complex numbers, integers and nonzero integers, respectively. All algebras and vector spaces considered in the paper are over $\C$.

It is well-known that the derivation algebra of an  algebra plays an important role in the study of the structures of the algebra. As a generalization of derivation, $\check{\text{S}}$emrl introduced the notion of 2-local derivation in \cite{Se},
2-local derivations on the algebra $B(H)$ of all bounded linear operators on the infinite-dimensional separable Hilbert space $H$ were determined therein. The concept of 2-local derivation is actually an important and meaningful property  in its own
right and have been widely studied by many scholars. For a given algebra $\mathcal{L}$, the main problem in this subject is to prove that they
automatically become derivations or to give examples of 2-local derivations of $\mathcal{L}$, which are not derivations. In \cite{ AKR}, it was shown that each 2-local derivation on a finite dimensional semisimple Lie algebra over an algebraically closed field of characteristic zero is a derivation and each finite dimensional nilpotent Lie algebra with dimension larger than two admits a 2-local derivation which is not a derivation. The authors in \cite{WN} proved that all 2-local superderivations on finite dimensional basic classical Lie superalgebras except $A(n,n)$ over an algebraically closed field of characteristic zero are derivations. The idea was exploited and generalized to consider 2-local derivation on some other infinite dimensional Lie (super) algebras, such as the Witt algebras and some of their subalgebras \cite{ AKY,ZCZ}, $W(2, 2)$ algebra  \cite{Ta}, twisted Heisenberg-Virasoro algebra \cite{ZC}, super Virasoro algebra, super $W(2,2)$ algebra \cite{DSL} and  Lie superalgebra of Block type \cite{SYZ}. In addition, the authors showed that every 2-local derivation on the simple Jacobson-Witt algebras over a field $\mathbb{F}$ of prime characteristic $p$ is a derivation in \cite{YZ}.

Galilean conformal algebras were initially introduced by Bagchi  and Gopakumar  in \cite{BG},  where the authors proposed these algebras as a different non-relativistic limit of the AdS/CFT conjecture and studied a non-relativistic conformal symmetry obtained by a parametric contraction of the relativistic conformal group. It was found that the finite  Galilean conformal algebra could be given an infinite-dimensional lift for all space-time dimensions (cf. ~\cite{BG, BG2, HR, MT}). These infinite-dimensional extensions contain a subalgebra isomorphic to the (centerless) Virasoro algebra. In particular, the infinite-dimensional Galilean conformal algebra in 2D turned out to be related to the symmetries of non-relativistic hydrodynamic equations \cite{A}, the BMS/GCA correspondence \cite{B1}, the tensionless limit of string theory \cite{B2} and statistical mechanics \cite{HSSU}. It is clear that the infinite-dimensional Galilean conformal algebra in $(1+1)$ dimensional space-time is the well-known $W(2,2)$ algebra \cite{ZD}. The infinite-dimensional Galilean conformal algebra in $(2+1)$ dimensional space-time, named the planar Galilean conformal algebra $\mathcal{G}$ by Aizawa in \cite{A}, is a Lie algebra with a basis $\{\mathbf{L}_m, \mathbf{H}_m, \mathbf{I}_m, \mathbf{J}_m\mid m\in\Z\}$ and the nontrivial Lie brackets defined by
\begin{equation*}
\aligned
&[\mathbf{L}_n,\mathbf{L}_m]=(m-n)\mathbf{L}_{m+n},\quad [\mathbf{L}_n,\mathbf{H}_m]=m\mathbf{H}_{m+n},\\
&[\mathbf{L}_n,\mathbf{I}_m]=(m-n)\mathbf{I}_{m+n},\quad [\mathbf{L}_n,\mathbf{J}_m]=(m-n)\mathbf{J}_{m+n},\\
&[\mathbf{H}_n,\mathbf{I}_m]=\mathbf{J}_{m+n},\quad [\mathbf{H}_n,\mathbf{J}_m]=-\mathbf{I}_{m+n},\quad \forall m,n\in\Z.
\endaligned
\end{equation*}
Set\begin{equation*}\label{bstr} L_m:=\mathbf{L}_m,\, H_m:=\sqrt{-1}\mathbf{H}_m,\, I_m:=\mathbf{I}_m+\sqrt{-1}\mathbf{J}_m\,\text{ and }\,J_m:=\mathbf{I}_m-\sqrt{-1}\mathbf{J}_m\,\text{ for }\, m\in\Z.\end{equation*}
Then $\{L_m, H_m, I_m, J_m\mid m\in\Z\}$ is another basis of $\mathcal{G}$  with the following nontrivial Lie brackets
\begin{equation*}
\aligned
&[L_n,L_m]=(m-n)L_{m+n},\quad [L_n,H_m]=mH_{m+n},\\
&[L_n,I_m]=(m-n)I_{m+n},\quad [L_n,J_m]=(m-n)J_{m+n},\\
&[H_n,I_m]=I_{m+n},\quad [H_n,J_m]=-J_{m+n},\quad \forall m,n\in\Z.
\endaligned
\end{equation*}
Whittaker modules over the planar Galilean conformal algebra $\mathcal{G}$  were studied in \cite{CY}, where a sufficient and necessary condition for a Whittaker module to be simple was precisely given. Very recently, using simple modules over the finite-dimensional solvable Lie algebras, the authors  constructed  many simple restricted modules over $\mathcal{G}$ in \cite{GG}. The structure theory of $\mathcal{G}$ has also been developed, such as, low-dimensional cohomology groups, biderivations, linear commuting maps, left-symmetric algebra structures  and Lie conformal algebras of  $\mathcal{G}$ were  studied in  \cite{GLP}, \cite{CSY}, \cite{CS} and \cite{HWX}, respectively. The aim of this paper is to study 2-local derivations on the planar Galilean conformal algebra $\mathcal{G}$.

The paper is organized as follows. In Section 2, we recall some fundamental definitions and basic known results that we need in the following. In Section 3, we prove that every 2-local derivation on the planar Galilean conformal algebra $\mathcal{G}$ is automatically a derivation.
\section{Preliminaries}
In this section, we give some necessary definitions and preliminary results.

A {\bf derivation} on a Lie algebra $\mathcal{L}$  is a linear transformation $D:\mathcal{L}\rightarrow\mathcal{L}$ such that the following Leibniz law holds:
$$D([x, y])=[D(x), y]+[x, D(y)],\,\,\forall\,x,y\in\mathcal{L}.$$
The set of all derivations of $\mathcal{L}$ is denoted by $\mathrm{Der}(\mathcal{L})$, which is a Lie algebra under the usual commutant operation. For each $x\in\mathcal{L}$, let
$$\mathrm{ad} (x):\mathcal{L}\rightarrow\mathcal{L},\,\,\,\mathrm{ad} (x) (y)=[x, y], \,\,\forall\,y\in\mathcal{L}.$$ Then $\mathrm{ad}(x)$ is a derivation on $\mathcal{L}$ for any $x\in\mathcal{L}$, which is  called an inner derivation. The set of all inner derivations of $\mathcal{L}$ is denoted by $\mathrm{Inn}(\mathcal{L})$, which is an ideal of $\mathrm{Der}(\mathcal{L})$.

Recall that a map $\Delta: \mathcal{L}\rightarrow\mathcal{L}$ (not necessarily linear) is called  a {\bf 2-local derivation} if for any $x,y\in\mathcal{L}$, there exists a derivation $\Delta_{x,y}\in\mathrm{Der}(\mathcal{L})$ (depending on $x, y$) such that $\Delta(x)=\Delta_{x,y}(x)$ and $\Delta(y)=\Delta_{x,y}(y)$.
In particular, for any $x\in\mathcal{L}$ and $k\in\C$, there exists $\d_{kx,x}\in\mathrm{Der}(\mathcal{L})$ such that
$$\Delta(kx)=\d_{kx,x}(kx)=k\d_{kx,x}(x)=k\Delta(x).$$

We now recall and establish several auxiliary results.
\begin{lemm}\label{prop1}(cf. \cite{GLP})
$\mathrm{Der}(\mathcal{G})=\mathrm{ad}(\mathcal{G})\oplus\C D$, where
$D$ is an outer derivation defined by $$D(L_m)=D(H_m)=0, D(I_m)=I_m, D(J_m)=J_m, \ \ \ \ \forall m\in\Z.$$
\end{lemm}
As a direct consequence of Lemma \ref{prop1}, we have the following.
\begin{lemm}\label{prop2}
Let $\Delta$ be a 2-local derivation on the planar Galilean conformal algebra $\mathcal{G}$. Then for every $x,y\in \mathcal{G}$, there exists a derivation  $\Delta_{x,y}$ of $\mathcal{G}$ such that  $\Delta(x)=\Delta_{x,y}(x),\Delta(y)=\Delta_{x,y}(y)$ and it can be written as
$$\Delta_{x,y}=\mathrm{ad}(\sum_{k\in\Z}a_k(x,y)L_k+\sum_{k\in\Z}b_k(x,y)H_k+\sum_{k\in\Z}c_k(x,y)I_k+\sum_{k\in\Z}d_k(x,y)J_k)+\lambda(x,y)D,$$
where $\lambda, a_k,b_k,c_k,d_k (k\in\Z)$ are complex-valued functions on $\mathcal{G}\times\mathcal{G}$ and $D$ is given by Lemma \ref{prop1}.
\end{lemm}
\section{ 2-local derivations on $\mathcal{G}$}
In this section, we will determine all  2-local derivations on the planar Galilean conformal algebra $\mathcal{G}$.
\begin{lemm}\label{prop3}
Let $\Delta$ be a 2-local derivation on $\mathcal{G}$.
\begin{itemize}\parskip-3pt
  \item[(i)] For a given $i\in\Z$, if $\Delta(L_i)=0$, then for any $y\in \mathcal{G}$,
  \begin{equation*}\label{er1}\d_{L_i,y}=\mathrm{ad}(a_i(L_i,y)L_i+b_0(L_i,y)H_0+c_i(L_i,y)I_i+d_i(L_i,y)J_i)+\lambda(L_i,y)D.\end{equation*}
  \item[(ii)] If $\d(I_0+J_0)=0$, then for any $y\in \mathcal{G}$, we have
  \begin{equation*}\label{er2}\d_{I_0+J_0,y}=\mathrm{ad}(a_0(I_0+J_0,y)L_0+\sum_{k\in\Z}c_k(I_0+J_0,y)I_k+\sum_{k\in\Z}d_k(I_0+J_0,y)J_k).\end{equation*}
\end{itemize}
 \end{lemm}
\begin{proof}
From Lemma \ref{prop2}, we can first assume that
\begin{align*}
\d_{L_i,y}&=\mathrm{ad}(\sum_{k\in\Z}a_k(L_i,y)L_k+\sum_{k\in\Z}b_k(L_i,y)H_k+\sum_{k\in\Z}c_k(L_i,y)I_k+\sum_{k\in\Z}d_k(L_i,y)J_k)\\
& \ \ \ +\ \lambda(L_i,y)D,\\
\d_{I_0+J_0,y}&=\mathrm{ad}(\sum_{k\in\Z}a_k(I_0+J_0,y)L_k+\sum_{k\in\Z}b_k(I_0+J_0,y)H_k+\sum_{k\in\Z}c_k(I_0+J_0,y)I_k\\
& \ \ \ +\ \sum_{k\in\Z}d_k(I_0+J_0,y)J_k)+\lambda(I_0+J_0,y)D,
\end{align*}
where $\lambda, a_k,b_k,c_k,d_k (k\in\Z)$ are complex-valued functions on $\mathcal{G}\times\mathcal{G}$.

{\rm (i)}  When $\Delta(L_i)=0$, we have
\begin{align*}
\Delta(L_i)&=\d_{L_i,y}(L_i)\\
&=[\sum_{k\in\Z}a_k(L_i,y)L_k+\sum_{k\in\Z}b_k(L_i,y)H_k+\sum_{k\in\Z}c_k(L_i,y)I_k+\sum_{k\in\Z}d_k(L_i,y)J_k,L_i]\\
& \ \ \ +\ \lambda(L_i,y)D(L_i)\\
&=\sum_{k\in\Z}((i-k)a_k(L_i,y)L_{k+i}-kb_k(L_i,y)H_{k+i}+(i-k)c_k(L_i,y)I_{k+i}\\
& \ \ \ +\ (i-k)d_k(L_i,y)J_{k+i})=0.
\end{align*}
This means
$$(i-k)a_k(L_i,y)=kb_k(L_i,y)=(i-k)c_k(L_i,y)=(i-k)d_k(L_i,y)=0, \ \ \ \ \forall k\in\Z,$$
from which we obtain $a_k(L_i,y)=c_k(L_i,y)=d_k(L_i,y)=0$ for all $k\in\Z$ with $k\neq i$ and $b_k(L_i,y)=0$ for all $k\in\Z^*$. So {\rm (i)} holds.

{\rm (ii)} When $\d(I_0+J_0)=0$, we compute
\begin{align*}
\Delta(I_0+J_0)&=\d_{I_0+J_0,y}(I_0+J_0)\\
&=[\sum_{k\in\Z}a_k(I_0+J_0,y)L_k+\sum_{k\in\Z}b_k(I_0+J_0,y)H_k+\sum_{k\in\Z}c_k(I_0+J_0,y)I_k\\
& \ \ \ +\ \sum_{k\in\Z}d_k(I_0+J_0,y)J_k,I_0+J_0]+\lambda(I_0+J_0,y)D(I_0+J_0)\\
&=\sum_{k\in\Z}(-ka_k(I_0+J_0,y)+b_k(I_0+J_0,y))I_k\\
& \ \ \ -\ \sum_{k\in\Z}(ka_k(I_0+J_0,y)+b_k(I_0+J_0,y))J_k+\lambda(I_0+J_0,y)(I_0+J_0)=0.
\end{align*}
This implies
\begin{eqnarray*}
&b_0(I_0+J_0,y)+\lambda(I_0+J_0,y)=-b_0(I_0+J_0,y)+\lambda(I_0+J_0,y)=0,\\
&b_k(I_0+J_0,y)-ka_k(I_0+J_0,y)=b_k(I_0+J_0,y)+ka_k(I_0+J_0,y)=0, \ \ \ \ \forall k\in\Z^*,
\end{eqnarray*}
which immediately yield $\lambda(I_0+J_0,y)=0, a_k(I_0+J_0,y)=0$ for all $k\in\Z^*$ and $b_k(I_0+J_0,y)=0$ for all $k\in\Z$.
Thus, {\rm (ii)}  holds.
\end{proof}
\begin{lemm}\label{prop4}
Let $\Delta$ be a 2-local derivation on $\mathcal{G}$ such that $\d(L_0)=\d(L_1)=0$. Then $\d(L_i)=0$ for all $i\in\Z$.
 \end{lemm}
\begin{proof} Since $\d(L_0)=\d(L_1)=0$, it follows from Lemma \ref{prop3} (i) that for  any $y\in\mathcal{G}$, we can assume that
\begin{eqnarray*}
&\d_{L_0,y}=\mathrm{ad}(a_0(L_0,y)L_0+b_0(L_0,y)H_0+c_0(L_0,y)I_0+d_0(L_0,y)J_0)+\lambda(L_0,y)D,\\
&\d_{L_1,y}=\mathrm{ad}(a_1(L_1,y)L_1+b_0(L_1,y)H_0+c_1(L_1,y)I_1+d_1(L_1,y)J_1)+\lambda(L_1,y)D.
\end{eqnarray*}
By taking $y=L_i$ in the above two equations, we respectively get
\begin{align*}
\Delta(L_i)&=\d_{L_0,L_i}(L_i)=[a_0(L_0,L_i)L_0+b_0(L_0,L_i)H_0+c_0(L_0,L_i)I_0+d_0(L_0,L_i)J_0, L_i]\\
&=ia_0(L_0,L_i)L_i+ic_0(L_0,L_i)I_i+id_0(L_0,L_i)J_i
\end{align*}
and
\begin{align*}
\Delta(L_i)&=\d_{L_1,L_i}(L_i)=[a_1(L_1,L_i)L_1+b_0(L_1,L_i)H_0+c_1(L_1,L_i)I_1+d_1(L_1,L_i)J_1, L_i]\\
&=(i-1)a_1(L_1,L_i)L_{i+1}+(i-1)c_1(L_1,L_i)I_{i+1}+(i-1)d_1(L_1,L_i)J_{i+1},
\end{align*}
from which one has
\begin{align*}
& ia_0(L_0,L_i)L_i+ic_0(L_0,L_i)I_i+id_0(L_0,L_i)J_i+(1-i)a_1(L_1,L_i)L_{i+1}\\
& \ \ \ +\ (1-i)c_1(L_1,L_i)I_{i+1}+(1-i)d_1(L_1,L_i)J_{i+1}=0.
\end{align*}
So $a_0(L_0,L_i)=c_0(L_0,L_i)=d_0(L_0,L_i)=0$ with $i\neq0$ and $a_1(L_1,L_i)=c_1(L_1,L_i)=d_1(L_1,L_i)=0$ with $i\neq1$. As a result, $\d(L_i)=0$, as desired.
\end{proof}
\begin{lemm}\label{prop5}
Let $\Delta$ be a 2-local derivation on $\mathcal{G}$ such that $\d(L_i)=0$ for all $i\in\Z$. Then for any $x=\sum_{t\in\Z}(\a_tL_t+\b_t H_t+\gamma_t I_t+\delta_t J_t)\in \mathcal{G}$, we have
$$\d(x)=\mu_{x}\sum_{t\in\Z}(\gamma_t I_t-\delta_t J_t)+\nu_{x}\sum_{t\in\Z}(\gamma_t I_t+\delta_t J_t),$$
where $\mu_{x}$ and $\nu_{x}$ are both complex numbers depending on $x$.
 \end{lemm}
\begin{proof}
For $x=\sum_{t\in\Z}(\a_tL_t+\b_t H_t+\gamma_t I_t+\delta_t J_t)\in \mathcal{G}$, as $\d(L_i)=0$ for all $i\in\Z$, from Lemma \ref{prop3} {\rm (i)} we have
\begin{align*}
\Delta(x)&=\d_{L_i,x}(x)\\
&=[a_i(L_i,x)L_i+b_0(L_i,x)H_0+c_i(L_i,x)I_i+d_i(L_i,x)J_i, x]+\lambda(L_i,x)D(x)\\
&=\sum_{t\in\Z}(t-i)a_i(L_i,x)\a_tL_{t+i}+\sum_{t\in\Z}ta_i(L_i,x)\b_tH_{t+i}\\
& \ \ \ +\ \sum_{t\in\Z}((t-i)c_i(L_i,x)\a_t-c_i(L_i,x)\b_t+(t-i)a_i(L_i,x)\gamma_t)I_{i+t}\\
& \ \ \ +\ \sum_{t\in\Z}((t-i)d_i(L_i,x)\a_t+d_i(L_i,x)\b_t+(t-i)a_i(L_i,x)\delta_t)J_{i+t}\\
& \ \ \ +\  b_0(L_i,x)\sum_{t\in\Z}(\gamma_tI_t-\delta_tJ_t)+\lambda(L_i,x)\sum_{t\in\Z}(\gamma_tI_t+\delta_tJ_t).
\end{align*}
By taking enough different $i\in\Z$ in the above equation and, if necessary, let these $i's$ be large enough, we obtain that
$$\Delta(x)=b_0(L_i,x)\sum_{t\in\Z}(\gamma_tI_t-\delta_tJ_t)+\lambda(L_i,x)\sum_{t\in\Z}(\gamma_tI_t+\delta_tJ_t).$$
In addition, it is easy to see that $\mu_{x}:=b_0(L_i,x)$ and $\nu_{x}:=\lambda(L_i,x)$ are both independent on $i$. The proof of this lemma is completed.
\end{proof}
\begin{lemm}\label{prop6}
Let $\Delta$ be a 2-local derivation on $\mathcal{G}$ such that $\d(I_0+J_0)=0$ and $\d(L_i)=0$ for all $i\in\Z$. Then for any
$p\in\Z^*$ and $y\in\mathcal{G}$, there are $\xi_{p}^y, \eta_{p}^y,\varepsilon_{p}^y\in\C$ such that
$$\d_{L_p+I_{2p}+J_{2p},y}=\mathrm{ad}(\xi_{p}^yL_p+\eta_{p}^yI_p+\xi_{p}^yI_{2p}+\varepsilon_{p}^yJ_p+\xi_{p}^yJ_{2p}).$$
\end{lemm}
\begin{proof}
According to $\d(L_i)=0$ for all $i\in\Z$ and Lemma \ref{prop5}, we have
$$\d(L_p+I_{2p}+J_{2p})=\mu_{L_p+I_{2p}+J_{2p}}(I_{2p}-J_{2p})+\nu_{L_p+I_{2p}+J_{2p}}(I_{2p}+J_{2p}),$$
where $\mu_{L_p+I_{2p}+J_{2p}}, \nu_{L_p+I_{2p}+J_{2p}}\in\C$. Combining $\d(I_0+J_0)=0$ with Lemma \ref{prop3} {\rm (ii)}, we know that
\begin{align*}
\Delta(L_p+I_{2p}+J_{2p})&=\d_{I_0+J_0,L_p+I_{2p}+J_{2p}}(L_p+I_{2p}+J_{2p})\\
&=[a_0(I_0+J_0,L_p+I_{2p}+J_{2p})L_0+\sum_{k\in\Z}c_k(I_0+J_0,L_p+I_{2p}+J_{2p})I_k\\
& \ \ \ +\ \sum_{k\in\Z}d_k(I_0+J_0,L_p+I_{2p}+J_{2p})J_k, L_p+I_{2p}+J_{2p}]\\
&=pa_0(I_0+J_0,L_p+I_{2p}+J_{2p})L_p+\sum_{k\in\Z}(p-k)c_k(I_0+J_0,L_p+I_{2p}+J_{2p})I_{p+k}\\
& \ \ \ +\ \sum_{k\in\Z}(p-k)d_k(I_0+J_0,L_p+I_{2p}+J_{2p})J_{p+k}\\
& \ \ \ +\ 2pa_0(I_0+J_0,L_p+I_{2p}+J_{2p})I_{2p}+2pa_0(I_0+J_0,L_p+I_{2p}+J_{2p})J_{2p}.
\end{align*}
Equating the above two expressions for $\Delta(L_p+I_{2p}+J_{2p})$ and comparing the coefficients of $L_p, I_{2p}$ and $J_{2p}$ of both sides, we immediately get $\mu_{L_p+I_{2p}+J_{2p}}=\nu_{L_p+I_{2p}+J_{2p}}=0$. Thus,
\begin{equation}\label{ac}\Delta(L_p+I_{2p}+J_{2p})=0.\end{equation}
For any $y\in\mathcal{G}$, by Lemma \ref{prop2}, we can assume that
\begin{align*}
\Delta_{L_p+I_{2p}+J_{2p},y}&=\mathrm{ad}(\sum_{k\in\Z}a_k(L_p+I_{2p}+J_{2p},y)L_k+\sum_{k\in\Z}b_k(L_p+I_{2p}+J_{2p},y)H_k\\
& \ \ \ +\ \sum_{k\in\Z}c_k(L_p+I_{2p}+J_{2p},y)I_k+\sum_{k\in\Z}d_k(L_p+I_{2p}+J_{2p},y)J_k\\
& \ \ \ +\ \lambda(L_p+I_{2p}+J_{2p},y)D.
\end{align*}
Using this and \eqref{ac}, we obtain
\begin{align}\label{opp}
&\Delta(L_p+I_{2p}+J_{2p})\nonumber\\
&=\Delta_{L_p+I_{2p}+J_{2p},y}(L_p+I_{2p}+J_{2p})\nonumber\\
&=[\sum_{k\in\Z}a_k(L_p+I_{2p}+J_{2p},y)L_k+\sum_{k\in\Z}b_k(L_p+I_{2p}+J_{2p},y)H_k\nonumber\\
& \ \ \ +\ \sum_{k\in\Z}c_k(L_p+I_{2p}+J_{2p},y)I_k+\sum_{k\in\Z}d_k(L_p+I_{2p}+J_{2p},y)J_k,L_p+I_{2p}+J_{2p}]\nonumber\\
& \ \ \ +\ \lambda(L_p+I_{2p}+J_{2p},y)D(L_p+I_{2p}+J_{2p})\nonumber\\
&=\sum_{k\in\Z}a_k(L_p+I_{2p}+J_{2p},y)((p-k)L_{p+k}+(2p-k)I_{2p+k}+(2p-k)J_{2p+k})\nonumber\\
& \ \ \ +\ \sum_{k\in\Z}b_k(L_p+I_{2p}+J_{2p},y)(-kH_{p+k}+I_{2p+k}-J_{2p+k})\nonumber\\
& \ \ \ +\ \sum_{k\in\Z}(p-k)(c_k(L_p+I_{2p}+J_{2p},y)I_{p+k}+d_k(L_p+I_{2p}+J_{2p},y)J_{p+k})\nonumber\\
& \ \ \ +\ \lambda(L_p+I_{2p}+J_{2p},y)(I_{2p}+J_{2p})=0.
\end{align}
From this, we deduce that $(p-k)a_k(L_p+I_{2p}+J_{2p},y)=0$  and $kb_k(L_p+I_{2p}+J_{2p},y)=0$, which in turn imply
$a_k(L_p+I_{2p}+J_{2p},y)=0$ for $k\neq p$ and $b_k(L_p+I_{2p}+J_{2p},y)=0$  for $k\neq 0$. Using these and observing the coefficients of $I_{2p}, J_{2p}$ in \eqref{opp}, we respectively get
\begin{equation*}
\aligned
&b_0(L_p+I_{2p}+J_{2p},y)+\lambda(L_p+I_{2p}+J_{2p},y)=0,\\
&-b_0(L_p+I_{2p}+J_{2p},y)+\lambda(L_p+I_{2p}+J_{2p},y)=0,
\endaligned
\end{equation*}
forcing $b_0(L_p+I_{2p}+J_{2p},y)=\lambda(L_p+I_{2p}+J_{2p},y)=0$. Similarly, by observing the coefficients of $I_{3p}, J_{3p}$ in \eqref{opp}, we respectively have
\begin{equation*}
\aligned
&-pc_{2p}(L_p+I_{2p}+J_{2p},y)+pa_p(L_p+I_{2p}+J_{2p},y)=0,\\
&-pd_{2p}(L_p+I_{2p}+J_{2p},y)+pa_p(L_p+I_{2p}+J_{2p},y)=0,
\endaligned
\end{equation*}
which yield $a_p(L_p+I_{2p}+J_{2p},y)=c_{2p}(L_p+I_{2p}+J_{2p},y)=d_{2p}(L_p+I_{2p}+J_{2p},y)$. Moreover, by observing the coefficients of $I_{k}, J_{k}, k\neq 2p,3p$ in \eqref{opp}, we obtain $c_k(L_p+I_{2p}+J_{2p},y)=d_k(L_p+I_{2p}+J_{2p},y)=0$ for all $k\neq p,2p$.
Denote $a_p(L_p+I_{2p}+J_{2p},y)=\xi_{p}^y, c_p(L_p+I_{2p}+J_{2p},y)=\eta_{p}^y$ and $d_p(L_p+I_{2p}+J_{2p},y)=\varepsilon_{p}^y$. This completes the proof.
\end{proof}
\begin{lemm}\label{prop7}
Let $\Delta$ be a 2-local derivation on $\mathcal{G}$ such that $\d(L_0)=\d(L_1)=\d(I_0+J_0)=0$. Then $\d(x)=0$ for all $x\in\mathcal{G}$.
\end{lemm}
\begin{proof}
Take any but fixed $x=\sum_{t\in\Z}(\a_tL_t+\b_t H_t+\gamma_t I_t+\delta_t J_t)\in \mathcal{G}$, where $\a_t, \b_t, \gamma_t, \delta_t\in\C$ for any $t\in\Z$. It follows from $\d(L_0)=\d(L_1)=0$ and Lemma \ref{prop4} that
\begin{equation}\label{vb}\d(L_i)=0,\ \ \ \ \forall i\in\Z.
\end{equation}
This along with Lemma \ref{prop5} gives
\begin{align}\label{vbn}
\Delta(x)&=\d(\sum_{t\in\Z}(\a_tL_t+\b_t H_t+\gamma_t I_t+\delta_t J_t))\nonumber\\
&=\mu_{x}\sum_{t\in\Z}(\gamma_t I_t-\delta_t J_t)+\nu_{x}\sum_{t\in\Z}(\gamma_t I_t+\delta_t J_t)
\end{align}
for some $\mu_{x}, \nu_{x}\in\C$. For any $p\in\Z^*$, by \eqref{vb}, $\d(I_0+J_0)=0$ and Lemma \ref{prop6}, we have
$$\d_{L_p+I_{2p}+J_{2p},x}=\mathrm{ad}(\xi_{p}^xL_p+\eta_{p}^xI_p+\xi_{p}^xI_{2p}+\varepsilon_{p}^xJ_p+\xi_{p}^xJ_{2p})$$
for some $\xi_{p}^x, \eta_{p}^x,\varepsilon_{p}^x\in\C$. Now we compute
\begin{align}\label{vnm}
\Delta(x)&=\Delta_{L_p+I_{2p}+J_{2p},x}(x)\nonumber\\
&=[\xi_{p}^xL_p+\eta_{p}^xI_p+\xi_{p}^xI_{2p}+\varepsilon_{p}^xJ_p+\xi_{p}^xJ_{2p}, \sum_{t\in\Z}(\a_tL_t+\b_t H_t+\gamma_t I_t+\delta_t J_t)]\nonumber\\
&=\sum_{t\in\Z}(t-p)\a_t\xi_{p}^xL_{p+t}+\sum_{t\in\Z}t\b_t\xi_{p}^xH_{p+t}\nonumber\\
& \ \ \ +\ \sum_{t\in\Z}((t-p)\a_t\eta_{p}^x-\b_t\eta_{p}^x+(t-p)\gamma_t\xi_{p}^x)I_{p+t}\nonumber\\
& \ \ \ +\ \sum_{t\in\Z}((t-p)\a_t\varepsilon_{p}^x+\b_t\varepsilon_{p}^x+(t-p)\delta_t\xi_{p}^x)J_{p+t}\nonumber\\
& \ \ \ +\ \sum_{t\in\Z}((t-2p)\a_t\xi_{p}^x-\b_t\xi_{p}^x)I_{2p+t}+\sum_{t\in\Z}((t-2p)\a_t\xi_{p}^x+\b_t\xi_{p}^x)J_{2p+t}.
\end{align}
Next, the proof is divided into the following two cases.
\begin{case}
$(\a_{t})_{t\in\Z}$ is not a zero sequence.
\end{case}
In this case, there is a nonzero term $\a_{t_0}L_{t_0}$ in $x=\sum_{t\in\Z}(\a_tL_t+\b_t H_t+\gamma_t I_t+\delta_t J_t)$ for some $t_0\in\Z$.
For any $p\in\Z^*$ in \eqref{vnm} with $p\neq t_0$, by equating  the coefficients of $L_{p+t_0}$ in \eqref{vbn} and \eqref{vnm}, we have $(t_0-p)\a_{t_0}\xi_{p}^x=0$, i.e., $\xi_{p}^x=0$. By \eqref{vbn} and \eqref{vnm}, we have
\begin{align*}
\Delta(x)&=\mu_{x}\sum_{t\in\Z}(\gamma_t I_t-\delta_t J_t)+\nu_{x}\sum_{t\in\Z}(\gamma_t I_t+\delta_t J_t)\\
&=\sum_{t\in\Z}((t-p)\a_t\eta_{p}^x-\b_t\eta_{p}^x)I_{p+t}+\sum_{t\in\Z}((t-p)\a_t\varepsilon_{p}^x+\b_t\varepsilon_{p}^x)J_{p+t}.
\end{align*}
By taking enough different $p$ in the above equation and, if necessary, let these $p's$ be large enough, we get $\d(x)=0$.
\begin{case}
$(\a_{t})_{t\in\Z}$ is a zero sequence.
\end{case}
In this case, we have  $x=\sum_{t\in\Z}(\b_t H_t+\gamma_t I_t+\delta_t J_t)$.
\begin{subcase}
$(\b_{t})_{t\in\Z}$ is not a zero sequence.
\end{subcase}
Consider first the situation that there is a nonzero term $\b_{t_0}H_{t_0}$ in  $x=\sum_{t\in\Z}(\b_t H_t+\gamma_t I_t+\delta_t J_t)$ for some $t_0\in\Z^*$. For any $p\in\Z^*$ in \eqref{vnm} with $p\neq t_0$, by equating  the coefficients of $H_{p+t_0}$ in \eqref{vbn} and \eqref{vnm}, we have $t_0\b_{t_0}\xi_{p}^x=0$, forcing $\xi_{p}^x=0$. Putting \eqref{vbn} and \eqref{vnm} together gives
\begin{align*}
\Delta(x)&=\mu_{x}\sum_{t\in\Z}(\gamma_t I_t-\delta_t J_t)+\nu_{x}\sum_{t\in\Z}(\gamma_t I_t+\delta_t J_t)\\
&=-\sum_{t\in\Z}\b_t\eta_{p}^xI_{p+t}+\sum_{t\in\Z}\b_t\varepsilon_{p}^xJ_{p+t}.
\end{align*}
Similarly, we obtain $\d(x)=0$ by taking enough different $p$ in the above equation and, if necessary, let these $p's$ be large enough.
Assume now that $\b_t=0$ for all $t\in\Z^*$, i.e., $\b_0\neq0$. Then  \eqref{vbn} and \eqref{vnm} become
\begin{align*}
\Delta(x)&=\mu_{x}\sum_{t\in\Z}(\gamma_t I_t-\delta_t J_t)+\nu_{x}\sum_{t\in\Z}(\gamma_t I_t+\delta_t J_t)\\
&=-\b_0\eta_{p}^xI_{p}+\sum_{t\in\Z}(t-p)\gamma_t\xi_{p}^xI_{p+t}+\b_0\varepsilon_{p}^xJ_{p}\\
& \ \ \ +\ \sum_{t\in\Z}(t-p)\delta_t\xi_{p}^xJ_{p+t}-\b_0\xi_{p}^xI_{2p}+\b_0\xi_{p}^xJ_{2p}.
\end{align*}
Also, we obtain $\d(x)=0$ by taking enough different $p$ in the above equation and, if necessary, let these $p's$ be large enough.
\begin{subcase}
$(\b_{t})_{t\in\Z}$ is a zero sequence.
\end{subcase}
In this case, we have $x=\sum_{t\in\Z}(\gamma_t I_t+\delta_t J_t)$. Equation  \eqref{vbn} along with  \eqref{vnm} becomes
\begin{align*}
\Delta(x)&=\mu_{x}\sum_{t\in\Z}(\gamma_t I_t-\delta_t J_t)+\nu_{x}\sum_{t\in\Z}(\gamma_t I_t+\delta_t J_t)\\
&=\sum_{t\in\Z}(t-p)\xi_{p}^x(\gamma_t I_{p+t}+\delta_t J_{p+t}).
\end{align*}
One can immediately get $\d(x)=0$ by taking enough different $p$ in the above equation and, if necessary, let these $p's$ be large enough. The proof is completed.
\end{proof}
Now we can formulate our main result in this section.
\begin{theo}\label{nm}
Every  2-local derivation on the planar Galilean conformal algebra $\mathcal{G}$ is a derivation.
\end{theo}
\begin{proof}
Let $\d$ be a 2-local derivation on $\mathcal{G}$. Take a derivation $\d_{L_0,L_1}$ such that
$$\d(L_0)=\d_{L_0,L_1}(L_0) \quad {\rm and}\quad \d(L_1)=\d_{L_0,L_1}(L_1).$$
Set $\d_1=\d-\d_{L_0,L_1}$. Then $\d_1$ is a 2-local derivation such that $\d_1(L_0)=\d_1(L_1)=0$. It follows from Lemma  \ref{prop4} that $\d_1(L_i)=0$ for all $i\in\Z$. Combined this with Lemma \ref{prop5}, we have $\d_1(I_0+J_0)=\mu_{I_0+J_0}(I_0-J_0)+\nu_{I_0+J_0}(I_0+J_0)$
for some $\mu_{I_0+J_0},\nu_{I_0+J_0}\in\C$. Now we set $\d_2=\d_1-\mu_{I_0+J_0}{\rm ad}(H_0)-\nu_{I_0+J_0}D$, then $\d_2$ is a 2-local derivation such that
\begin{eqnarray*}
&\d_2(L_0)=\d_1(L_0)-\mu_{I_0+J_0}[H_0, L_0]-\nu_{I_0+J_0}D(L_0)=0,\\
&\d_2(L_1)=\d_1(L_1)-\mu_{I_0+J_0}[H_0, L_1]-\nu_{I_0+J_0}D(L_1)=0, \\
&\d_2(I_0+J_0)=\d_1(I_0+J_0)-\mu_{I_0+J_0}[H_0, I_0+J_0]-\nu_{I_0+J_0}D(I_0+J_0)\\
&=\mu_{I_0+J_0}(I_0-J_0)+\nu_{I_0+J_0}(I_0+J_0)-\mu_{I_0+J_0}(I_0-J_0)-\nu_{I_0+J_0}(I_0+J_0)=0.\end{eqnarray*}
By Lemma \ref{prop7}, we see that $\d_2=\d-\d_{L_0,L_1}-\mu_{I_0+J_0}{\rm ad}(H_0)-\nu_{I_0+J_0}D\equiv0$. Thus $\d=\d_{L_0,L_1}+\mu_{I_0+J_0}{\rm ad}(H_0)+\nu_{I_0+J_0}D$ is a derivation, completing the proof.
\end{proof}


\end{document}